\documentclass[11pt]{article}

\usepackage{amsmath,amssymb,amsthm,mathtools,mathrsfs}
\usepackage{microtype}
\usepackage[colorlinks=true,linkcolor=blue,citecolor=blue,urlcolor=blue]{hyperref}
\usepackage[nameinlink,noabbrev]{cleveref}
\usepackage{enumitem}

\newcommand{\OmegaD}{\Omega}
\newcommand{\R}{\mathbb R}

\newcommand{\E}{\mathbb E}
\newcommand{\Prob}{\mathbb P}
\newcommand{\Var}{\operatorname{Var}}
\newcommand{\Cov}{\operatorname{Cov}}
\newcommand{\Law}{\operatorname{Law}}
\newcommand{\supp}{\operatorname{supp}}

\newcommand{\la}{\langle}
\newcommand{\ra}{\rangle}
\newcommand{\tr}{\operatorname{tr}}

\theoremstyle{plain}
\newtheorem{theorem}{Theorem}[section]
\newtheorem{proposition}[theorem]{Proposition}
\newtheorem{lemma}[theorem]{Lemma}
\newtheorem{corollary}[theorem]{Corollary}

\theoremstyle{definition}

\theoremstyle{remark}
\newtheorem{remark}[theorem]{Remark}

\title{A note on several inverse problems with generally random coefficients}
\author{C\u{a}t\u{a}lin I. C\^{a}rstea\thanks{Department of Applied Mathematics, National Yang Ming Chiao Tung University, Hsinchu 30050, Taiwan. Email: \href{mailto:catalin.carstea@gmail.com}{\texttt{catalin.carstea@gmail.com}}.}}
\date{}

\begin{document}
\maketitle

\begin{abstract}
We consider several inverse problems for elliptic equations whose coefficients are random, without imposing a special probabilistic structure on the randomness.  The main body treats the Schr\"odinger equation.  We compare what can be recovered from the full law of the Dirichlet-to-Neumann map, from its expectation, from finitely many joint moments of its boundary bilinear form, and from the averaged interior Green's operator.  We obtain both positive and negative results.  That the full law of the Dirichlet-to-Neumann map determines the law of the random potential is almost trivial.  However, the expected Dirichlet-to-Neumann map and, more generally, any fixed finite hierarchy of its boundary moments need not determine even the mean potential.  In contrast, the averaged Schr\"odinger Green's operator determines the pointwise mean and variance of the potential.  In a two-atom model it determines all pointwise moments of the two-point law.    The appendices contain the corresponding results for the conductivity equation.
\end{abstract}

\section{Introduction}

This note is concerned with inverse problems for elliptic equations whose coefficients are random.  The randomness is mostly not assumed to have a special form: it is not required to be Gaussian, perturbative, microlocally isotropic, or supported on inclusions.  We ask what can be recovered in this generally random setting from several types of data, from knowing the full probability distribution of, for example, the boundary data to knowing just its expectation.

The main body of the paper treats the Schr\"odinger equation on a bounded domain $\OmegaD\subset\R^n$.  For a deterministic potential $q$, let $u_f$ solve
\[
  (-\Delta+q)u_f=0\quad\text{in }\OmegaD,
  \qquad u_f|_{\partial\OmegaD}=f,
\]
and let $\Lambda_q f=\partial_\nu u_f|_{\partial\OmegaD}$ be the Dirichlet-to-Neumann map.  We also write
\[
  G_q=(-\Delta_D+q)^{-1}
\]
for the Dirichlet Green's operator.  Thus $\Lambda_q$ is the boundary map in the classical Calder\'on-type inverse problem for the Schr\"odinger operator, while $G_q$ is an interior source-to-solution operator.  If $q=q(\omega,x)$ is random, the following objects carry different amounts of information:
\[
  \Law(\Lambda_q),\qquad
  \E\Lambda_q,
  \qquad
  \text{finitely many joint moments of }\la\Lambda_qf,g\ra,
  \qquad
  \E G_q.
\]
 The purpose of the paper is to illustrate their differences.

\subsubsection*{Results}

The first result is an obstruction for averaged Dirichlet-to-Neumann maps.  The expected Dirichlet-to-Neumann map of a genuinely random potential can be indistinguishable from the Dirichlet-to-Neumann map of a deterministic potential which is not the mean potential.  More precisely, we will construct smooth nonnegative radial potentials $q,q_1,q_2$ in the unit ball of $\R^3$ such that
\begin{equation}\label{eq:intro-DN-half}
  \Lambda_q=\frac12\Lambda_{q_1}+\frac12\Lambda_{q_2}.
\end{equation}
If $Q$ takes the values $q_1$ and $q_2$ with probability $1/2$, then $\E\Lambda_Q=\Lambda_q$, but $q\ne \E Q$.  It follows then that it is not possible for $\E\Lambda_Q$ to determine $\E q$, as it does in the deterministic case.

This obstruction is not confined to first moments.  For every fixed integer $M$, we can construct finitely valued random potentials $q_A$ and $q_B$ for which all joint moments of the scalar quantities $\la\Lambda_qf,g\ra$ up to order $M$ agree, for all boundary functions appearing in the moments, but
\[
  \E q_A\ne \E q_B.
\]
In particular, even knowing both the expectation and the covariance form of the random Dirichlet-to-Neumann map does not determine the mean potential in the general class considered here.

The second result shows that this loss of information is caused by averaging and finite moment truncation, not by a failure of deterministic uniqueness.  If the full law of the random Dirichlet-to-Neumann map is known, then deterministic uniqueness lifts directly to recovery of probability laws.  We formulate this by recording the joint law of countably many matrix entries $\la\Lambda_qf_i,f_j\ra$ with respect to a dense sequence of boundary functions.  On any deterministic coefficient class where $q\mapsto\Lambda_q$ is injective and the corresponding coordinate map is Borel, this joint law determines $\Law(q)$.  In particular, classical deterministic uniqueness theorems for the Schr\"odinger inverse problem, such as the theorem of Sylvester and Uhlmann in dimensions $n\ge3$ \cite{SylvesterUhlmann1987}, immediately imply full-law recovery on their coefficient classes.

The third result concerns the averaged Green's operator (the ``source-to-solution'' map).  The full interior symbol of $G_q$ contains local information about $q$, and the corresponding symbol expansion for $\E G_q$ retains local moments of the potential in a way that the expected Dirichlet-to-Neumann map need not.  We show that the full interior symbol of $\E G_q$ determines $\E q$ and $\Var(q)$ pointwise.  For a two-valued random potential, the same triangular symbolic structure determines all pointwise moments of the two-point law.  We also treat a finite-dimensional model
\[
  q(\omega,x)=q_0(x)+\sum_{j=1}^d X_j(\omega)V_j(x),
\]
where the profiles $q_0,V_1,\ldots,V_d$ are known.  Under a natural algebraic nondegeneracy condition on the profiles, the averaged Green's operator determines all mixed moments, and hence the law, of the compactly supported random vector $X$.

The two-atom statements also have a deterministic interpretation.  They are not only examples about random coefficients, but also statements about the geometry of the coefficient-to-measurement correspondences.  The boundary counterexample gives a nontrivial affine identity in the range of $q\mapsto\Lambda_q$: a convex combination of two deterministic DN maps is again a deterministic DN map, although not the one corresponding to the same convex combination of the potentials.  This is a range-characterization issue, not just a probabilistic one.  By contrast, the two-atom rigidity result for averaged Green's operators says that the analogous affine identity for $q\mapsto G_q$ is trivial.  Thus the comparison between boundary maps and interior Green's operators may also be read as a comparison between two deterministic nonlinear ranges.  Deterministic range and convexity questions for Calder\'on-type coefficient maps have also appeared in related forms, for example in the convexity result for Schr\"odinger DN maps near the zero potential in \cite{CarsteaFeizmohammadiOksanen2023} and in finite-dimensional monotonicity-based formulations of the Calder\'on problem \cite{Harrach2023}.

The appendices record the corresponding conductivity statements.  The boundary-map counterexamples have conductivity analogues, obtained from the radial Schr\"odinger construction by the Liouville transform.  For conductivity Green's operators the principal symbol sees $\E(\gamma^{-1})$ rather than $\E\gamma$, and the first nontrivial correction determines a weighted covariance matrix involving logarithmic gradients of $\gamma^{-1}$.

\subsubsection*{Comparison to existing work}

A useful comparison point is the literature on inverse scattering for random potentials.  Lassas, P\"aiv\"arinta and Saksman studied inverse problems for random potentials in the Schr\"odinger equation in \cite{LassasPaivarintaSaksman2004,LassasPaivarintaSaksman2008}.  In the two-dimensional random scattering result of \cite{LassasPaivarintaSaksman2008}, the potential is Gaussian and its covariance operator is a classical pseudodifferential operator; from one realization of the backscattered field, the principal symbol of the covariance operator is recovered.  Caro, Helin and Lassas studied the corresponding plane-wave backscattering problem in dimension $n\ge 2$ for Gaussian microlocally isotropic random fields.  In general dimensions their result recovers the local strength from the single-scattering, or first Born, contribution, while in the case $n=3$ and covariance order $-3$ they prove the corresponding full nonlinear inverse backscattering result \cite{CaroHelinLassas2016}.  Related random-potential scattering results include recent stability, one-dimensional Helmholtz, elastic, biharmonic, and polyharmonic variants \cite{WangXuZhao2025StabilityPotential,WangXuZhao2025OneDim,LiLiWang2022ElasticPotential,LiWang2024BiharmonicRandomPotential,LiLiWangYang2025PolyharmonicPotential}.  Ma's survey gives a broader account of single-realization recovery results for random Schr\"odinger systems \cite{Ma2021}.

These papers do not assume an arbitrary probability law of the potential, and do not recover the pointwise mean of a coefficient.  What they recover is the local strength of a microlocally isotropic Gaussian random field, equivalently the principal symbol of the covariance operator.  Thus the assumed microlocal covariance structure is part of the identifiable object.  If one leaves this class, or if one asks for lower-order covariance information, a smooth covariance kernel, the mean, or the full probability law, the quoted theorems do not by themselves give such a recovery result.  At the level of leading singularities, this is analogous to the deterministic recovery-of-singularities program, where singular information about a potential is recovered from backscattering data; see Greenleaf--Uhlmann \cite{GreenleafUhlmann1993} and Reyes--Ruiz \cite{ReyesRuiz2012}.

A related scattering literature concerns random sources, or simultaneous source and potential questions.  For Schr\"odinger equations with unknown source and potential terms, Li, Liu and Ma proved uniqueness results for the variance of the source, the potential, and the expectation of the source from far-field measurements; later work treats the case where both the potential and source are random, and the case of a deterministic unknown potential with a microlocally isotropic Gaussian random source \cite{LiLiuMa2019,LiLiuMa2021,LiuMa2023}.  Random source scattering has also been studied for acoustic, elastic, attenuated Helmholtz, and related wave models, where the recovered quantities are typically statistical properties or principal symbols associated with the source covariance \cite{BaoChenLi2016,LiChenLi2018,LiWang2021HelmholtzAttenuation,LiHelinLi2020,LiLi2019ElasticSource}.  Passive imaging and correlation-based inverse problems form another adjacent direction: one uses fields generated by random or pseudorandom sources, or their long-time correlations, to recover metric, travel-time, Green's function, or covariance information \cite{HelinLassasOksanen2012,HelinLassasOksanen2014,HelinLassasOksanenSaksala2018,DeHoopSolna2009,BardosGarnierPapanicolaou2008}.

Another nearby line of work concerns electrical impedance tomography with stochastic conductivities.  Barth, Harrach, Hyv\"onen and Mustonen considered stochastic inclusions in a deterministic background and proved that the support of the inclusion can be detected from the mean Neumann-to-Dirichlet map by applying the factorization method or the monotonicity method, provided the random inclusion has sufficiently large contrast in the sense of expectation \cite{BarthHarrachHyvonenMustonen2017}.  That result assumes an inclusion geometry and a sign/contrast condition, and its conclusion is detection of the inclusion.  Stochastic homogenization and probabilistic interpretations of electrical impedance tomography have also been studied by Simon and by Piiroinen--Simon \cite{Simon2015,PiiroinenSimon2016,PiiroinenSimon2017}; these works are closer to stochastic forward modelling and probabilistic reformulations of Calder\'on's problem than to recovery of an arbitrary coefficient law from the law or moments of a boundary map.

There is also a distinct theory of inverse problems for stochastic partial differential equations.  In that setting the equation itself is stochastic, typically parabolic or hyperbolic, and the analysis often uses stochastic Carleman estimates to recover sources, coefficients, or states from observations of the random evolution.  The survey of L\"u and Zhang describes this program and emphasizes inverse problems that are genuinely stochastic and cannot simply be reduced to deterministic inverse problems \cite{LuZhang2023}. 

\subsubsection*{Organization of the paper}

 \Cref{sec:Schrodinger-setup} fixes the notation for Schr\"odinger DN maps, Green's operators, random potentials, and covariance forms.  \Cref{sec:schrodinger-DN-obstructions} proves the radial counterexamples for expected DN maps and finite DN moment hierarchies, and records a Jensen-type inequality for expected DN maps.  \Cref{sec:law-DN-map} proves the full-law recovery statement.  \Cref{sec:Green-potential} studies averaged Schr\"odinger Green's operators, including the two-atom and finite-dimensional models.  The appendices contain the conductivity analogues.

\section{Schr\"odinger DN maps, Green's operators, and random potentials}\label{sec:Schrodinger-setup}

We use the following conventions throughout the Schr\"odinger part of the paper, unless a statement explicitly gives a different coefficient class.  The domain $\OmegaD\subset\R^n$, $n\ge2$, is bounded with smooth boundary.  The default deterministic class is
\[
  \mathcal Q_+^\infty(\OmegaD)
  =\{q\in C^\infty(\overline\OmegaD;\R): q\ge0\}.
\]
For $q\in \mathcal Q_+^\infty(\OmegaD)$ the Dirichlet realization $-\Delta_D+q$ is strictly positive, hence
\[
  G_q=(-\Delta_D+q)^{-1}
\]
is well defined.  More generally, the same notation is used for real smooth potentials for which $0$ is not a Dirichlet eigenvalue, but this extra generality will not be needed in the counterexamples below.  For a boundary value $f$, let $u_f$ solve
\[
  (-\Delta+q)u_f=0\quad\text{in }\OmegaD,
  \qquad u_f|_{\partial\OmegaD}=f.
\]
The Schr\"odinger Dirichlet-to-Neumann map is
\[
  \Lambda_qf=\partial_\nu u_f|_{\partial\OmegaD}.
\]
We regard $\Lambda_q$ as the corresponding symmetric bilinear form on the real trace space $H^{1/2}(\partial\OmegaD)$, or equivalently as a bounded operator $H^{1/2}(\partial\OmegaD)\to H^{-1/2}(\partial\OmegaD)$.

A random potential means a strongly measurable map from a probability space $(\mathcal X,\mathcal F,\Prob)$ into the indicated coefficient class.  Expectations of operators are weak operator expectations.  Thus, if $T_\omega:X\to X^*$ is a random family of bounded operators, we write $\overline T=\E T_\omega$ provided the scalar functions $\omega\mapsto\la T_\omega f,g\ra$ belong to $L^1(\mathcal X)$ for all $f,g\in X$, and then define
\[
  \la \overline T f,g\ra=\E\la T_\omega f,g\ra .
\]
For DN maps we take $X=H^{1/2}(\partial\OmegaD)$, with the real bilinear or complex sesquilinear duality according to the scalar field.  For Green's operators we use the analogous weak expectation on $L^2(\OmegaD)$, or after localization by compactly supported cutoffs in the interior.  Higher probabilistic moments of matrix coefficients are used only when the corresponding scalar random variables have the required finite moments.  In particular, the finite-valued random potentials used in the counterexamples below satisfy this condition automatically.  When full interior symbols of averaged Green's operators are used, we impose the stated uniform $C^N$ bounds so that the usual interior parametrix construction may be averaged term by term.

For $q_1,q_2\in \mathcal Q_+^\infty(\OmegaD)$, Alessandrini's identity gives
\begin{equation}\label{eq:Alessandrini-q}
  \la(\Lambda_{q_1}-\Lambda_{q_2})f,g\ra
  =\int_\OmegaD (q_1-q_2)u_f^{(1)}u_g^{(2)}\,dx,
\end{equation}
where $u_f^{(j)}$ solves the equation with potential $q_j$.

For a random DN map we use the covariance of the associated random bilinear form whenever the scalar matrix coefficients have finite second moments.  More precisely, if
\[
  \omega\mapsto \la \Lambda_{q(\omega)}f,g\ra
  \quad\text{belongs to }L^2(\mathcal X)
\]
for the boundary values under consideration, set
\begin{equation}\label{eq:DN-covariance-definition}
\begin{split}
  \Cov(\Lambda_q)(f,g;h,k)
  :=\E\Big[&\la(\Lambda_q-\E\Lambda_q)f,g\ra \\
  &\times \la(\Lambda_q-\E\Lambda_q)h,k\ra\Big].
\end{split}
\end{equation}
For complex boundary values one inserts a complex conjugate in the second factor.

\section{Schr\"odinger DN maps: obstructions and order constraints}
\label{sec:schrodinger-DN-obstructions}

\subsection{Mean Schr\"odinger DN maps can coincide with deterministic DN maps}
\label{sec:DN-counterexamples}

We first record the obstruction for expected DN maps.  The examples are radial and explicit.  Although elementary, they show that $\E\Lambda_q$ alone cannot support a general recovery theorem for the mean coefficient.

\subsubsection{A radial Schr\"odinger counterexample}

We first isolate the elementary radial construction behind the examples.  The following notation will also be used below.

\begin{lemma}\label{lem:radial-affine-family}
Let $n\ge2$, let $\OmegaD=B(0,1)\subset\R^n$, and fix $\lambda>0$ and an integer $N\ge1$ with $\lambda>2N$.  Put
\begin{equation}\label{eq:QN-family}
  d_N=\frac{\lambda-2N}{\lambda+2N},
  \qquad
  Q_N(r)=\frac{8N^2d_Nr^{2N-2}}{(1-d_Nr^{2N})^2}.
\end{equation}
Then $Q_N$ is smooth and nonnegative on $\overline B(0,1)$; in fact $Q_N$ is real analytic in a neighborhood of $\overline B(0,1)$.  If $Y_{\ell}$ is a spherical harmonic of degree $\ell$ and
\[
  (-\Delta+Q_N)u=0,\qquad u|_{\partial B}=Y_{\ell},
\]
then $\Lambda_{Q_N}Y_{\ell}=\mu_{\ell,N}^{(n)}Y_{\ell}$, where
\begin{equation}\label{eq:mu-N-formula-general}
  \mu_{\ell,N}^{(n)}
  =\ell+\frac{\lambda^2-4N^2}{2(\lambda+2\ell+n-2)}.
\end{equation}
In particular, for fixed $\lambda$ the DN eigenvalues are affine functions of $N^2$.
\end{lemma}

\begin{proof}
Since $0<d_N<1$, the denominator in \eqref{eq:QN-family} is nonzero for $0\le r\le1$.  Moreover $Q_N$ is a rational function of $|x|^2$, namely
\[
  Q_N(x)=\frac{8N^2d_N(|x|^2)^{N-1}}{(1-d_N(|x|^2)^N)^2},
\]
so it is real analytic near $\overline B(0,1)$ and nonnegative.

The radial solution has the form $u(r,\omega)=R_\ell(r)Y_\ell(\omega)$, where
\[
 -R_\ell''-\frac{n-1}{r}R_\ell'
 +\frac{\ell(\ell+n-2)}{r^2}R_\ell+Q_N(r)R_\ell=0,
 \qquad R_\ell(1)=1,
\]
and $R_\ell(r)=O(r^\ell)$ at $r=0$.  Set $r=e^{-x}$,
\[
  \alpha=\frac{n-2}{2},
  \qquad
  \nu=\ell+\alpha,
  \qquad
  R_\ell(e^{-x})=e^{\alpha x}y(x).
\]
Then $y$ solves
\[
  -y''+V_Ny=-\nu^2y,
  \qquad
  V_N(x)=e^{-2x}Q_N(e^{-x})
  =\frac{8N^2d_Ne^{-2Nx}}{(1-d_Ne^{-2Nx})^2}.
\]
The regular solution is proportional to
\[
  y(x)=e^{-\nu x}
  \frac{1-c_Ne^{-2Nx}}{1-d_Ne^{-2Nx}},
  \qquad
  c_N=d_N\frac{\nu-N}{\nu+N}.
\]
Indeed, a direct substitution verifies the differential equation, and as $x\to\infty$ one has $y(x)\sim e^{-\nu x}$; therefore $R_\ell(r)\sim r^\ell$ as $r\to0$.  The boundary normalization $R_\ell(1)=1$ only multiplies this solution by a constant and does not affect the logarithmic derivative at $r=1$.  Moreover,
\[
  \frac{y'(0)}{y(0)}
  =-\nu+\frac{2Nc_N}{1-c_N}-\frac{2Nd_N}{1-d_N}
  =-\nu+\frac{(\lambda-2N)(\nu-N)}{\lambda+2\nu}
      -\frac{\lambda-2N}{2}.
\]
Since $R_\ell'(1)=-y'(0)/y(0)-\alpha$, this gives
\[
  R_\ell'(1)
  =\ell+\frac{\lambda^2-4N^2}{2(\lambda+2\nu)}
  =\ell+\frac{\lambda^2-4N^2}{2(\lambda+2\ell+n-2)}.
\]
This is \eqref{eq:mu-N-formula-general}.
\end{proof}

\begin{proposition}\label{prop:radial-schrodinger-counterexample}
Let $\OmegaD=B(0,1)\subset\R^3$.  Define
\begin{align*}
q_1(r)&=\frac{56/9}{\left(1-\frac79 r^2\right)^2},\\[3pt]
q(r)&=\frac{600}{13}\,\frac{r^8}{\left(1-\frac3{13}r^{10}\right)^2},\\[3pt]
q_2(r)&=\frac{392}{15}\,\frac{r^{12}}{\left(1-\frac1{15}r^{14}\right)^2}.
\end{align*}
Then $q,q_1,q_2$ are nonnegative real analytic functions in a neighborhood of $\overline{B(0,1)}$, they are pairwise distinct, and
\begin{equation}\label{eq:schrodinger-half-counterexample}
  \Lambda_q=\frac12\Lambda_{q_1}+\frac12\Lambda_{q_2}.
\end{equation}
\end{proposition}

\begin{proof}
Use Lemma~\ref{lem:radial-affine-family} with $n=3$ and $\lambda=16$.  The choices $N=1,5,7$ give the three displayed potentials $q_1=Q_1$, $q=Q_5$, and $q_2=Q_7$.  Since
\[
  5^2=\frac12\cdot1^2+\frac12\cdot7^2,
\]
formula \eqref{eq:mu-N-formula-general} gives
\[
  \mu_{\ell,5}^{(3)}=\frac12\mu_{\ell,1}^{(3)}+\frac12\mu_{\ell,7}^{(3)}
  \qquad\text{for every }\ell\ge0.
\]
The three potentials are pairwise distinct: $Q_1(0)>0$, whereas $Q_5$ and $Q_7$ vanish at the origin to different orders.  The DN maps are diagonal in the spherical harmonic basis, hence
\[
  \Lambda_{Q_5}=\frac12\Lambda_{Q_1}+\frac12\Lambda_{Q_7}.
\]
This is \eqref{eq:schrodinger-half-counterexample}.
\end{proof}

\begin{corollary}\label{cor:random-potential-mean-DN-deterministic}
There is a two-valued random potential $Q$ and a deterministic potential $q_*$ such that
$
  \E[\Lambda_Q]=\Lambda_{q_*},
$
while
$
  q_*\ne \E Q.
$
\end{corollary}

We also have the following.

\begin{corollary}\label{cor:barycentric-DN-identities}
Let $n\ge2$ and let $Q_N$ be the family in Lemma~\ref{lem:radial-affine-family}.  Suppose that $N_0,N_1,\dots,N_J$ are positive integers.  Let $p_1,\dots,p_J>0$ satisfy $\sum_{j=1}^Jp_j=1$, and assume that
\[
  N_0^2=\sum_{j=1}^Jp_jN_j^2,
  \qquad
  \lambda>2\max\{N_0,N_1,\dots,N_J\}.
\]
Then, on $B(0,1)\subset\R^n$,
\[
  \Lambda_{Q_{N_0}}=\sum_{j=1}^Jp_j\Lambda_{Q_{N_j}}.
\]
Thus the preceding counterexample is one member of a whole family of affine DN identities.
\end{corollary}

\begin{proof}
This follows immediately from the affine dependence on $N^2$ in \eqref{eq:mu-N-formula-general} and diagonalization in spherical harmonics.
\end{proof}

\subsection{No finite hierarchy of Schr\"odinger DN moments determines the mean potential}
\label{sec:finite-DN-moments}

The preceding examples show that $\E\Lambda_q$ does not determine $\E q$.  In fact, no finite hierarchy of probabilistic moments of the DN map determines the mean potential.

\begin{theorem}\label{thm:finite-DN-moments-not-enough}
Let $M\ge1$.  There exist two random potentials $q_A,q_B$ taking finitely many values among nonnegative radial potentials that are real analytic in a neighborhood of $\overline{B(0,1)}\subset\R^3$ such that, for every $1\le r\le M$ and all real boundary values
\[
  f_j,g_j\in H^{1/2}(\partial B(0,1);\R),
\]
one has
\begin{equation}\label{eq:finite-DN-moment-equality}
  \E\prod_{j=1}^r\la\Lambda_{q_A}f_j,g_j\ra
  =
  \E\prod_{j=1}^r\la\Lambda_{q_B}f_j,g_j\ra,
\end{equation}
but
\[
  \E q_A\ne\E q_B.
\]
The same construction gives the corresponding complex sesquilinear identities, including identities in which any prescribed factors are complex conjugated.  All expectations in these identities are finite.
\end{theorem}

\begin{proof}
Use the family $Q_N$ in Lemma~\ref{lem:radial-affine-family} with $n=3$.  Choose distinct positive integers
\[
  1=N_0<N_1<\cdots<N_{M+1}
\]
and then choose $\lambda>2\max_jN_j$.  Write $S_j=N_j^2$.  The $(M+1)\times(M+2)$ Vandermonde matrix with entries $S_j^r$, $0\le r\le M$, has rank $M+1$, because any $M+1$ columns form an invertible square Vandermonde matrix.  Hence its kernel is one-dimensional.  Moreover, no component of a nonzero kernel vector can vanish; otherwise the remaining $M+1$ components would give a nontrivial null vector for an invertible square Vandermonde matrix.

Let $c=(c_0,\dots,c_{M+1})$ be a nonzero vector in this kernel and decompose the signed measure
\[
  \sum_{j=0}^{M+1}c_j\delta_{N_j}
\]
into its positive and negative parts.  Since the zeroth moment vanishes, these two parts have the same total mass.  After normalization they define two probability laws for integer-valued random variables $N_A$ and $N_B$.  The corresponding random variables
\[
  S_A=N_A^2,
  \qquad
  S_B=N_B^2
\]
satisfy
\[
  \E S_A^m=\E S_B^m,
  \qquad 0\le m\le M.
\]

For fixed $\lambda$, formula \eqref{eq:mu-N-formula-general} gives
\[
  \mu_{\ell,N}^{(3)}
  =\ell+\frac{\lambda^2}{2(\lambda+2\ell+1)}
   -S\frac{2}{\lambda+2\ell+1},
  \qquad S=N^2.
\]
Thus
\[
  \Lambda_{Q_N}=A-SB,
\]
where $A$ and $B$ are fixed diagonal operators in the spherical harmonic basis, independent of $N$.  For each pair of real boundary values $f,g$,
\[
  \la\Lambda_{Q_N}f,g\ra
  =\alpha(f,g)-S\beta(f,g)
\]
with fixed bilinear forms $\alpha$ and $\beta$.  Therefore
\[
  \prod_{j=1}^r\la\Lambda_{Q_N}f_j,g_j\ra
\]
is a polynomial in $S$ of degree at most $r$.  If $r\le M$, the equality of the moments of $S_A$ and $S_B$ through order $M$ gives \eqref{eq:finite-DN-moment-equality} for
\[
  q_A=Q_{N_A},
  \qquad
  q_B=Q_{N_B}.
\]
The random potentials $q_A$ and $q_B$ take only finitely many smooth values.  Hence, for fixed boundary values, each scalar matrix coefficient $\la\Lambda_{q_A}f,g\ra$ or $\la\Lambda_{q_B}f,g\ra$ takes only finitely many complex values, and every finite product appearing in the stated moment identities has finite expectation.  The same polynomial argument applies over the complex trace space: a sesquilinear matrix coefficient is still affine in the real scalar $S$, and its complex conjugate is affine in $S$ as well.

The mean potentials are nevertheless different.  Indeed,
\[
  Q_1(0)=8\frac{\lambda-2}{\lambda+2}>0,
  \qquad
  Q_N(0)=0\quad(N\ge2).
\]
Because $c_0\ne0$, exactly one of the two probability laws assigns positive mass to $N_0=1$.  Hence $\E q_A(0)\ne\E q_B(0)$, possibly after interchanging the labels $A$ and $B$.
\end{proof}

\subsection{A Jensen inequality for mean Schr\"odinger DN maps}
\label{sec:DN-Jensen-potential}

The counterexamples above do not contradict the basic variational concavity of the DN map as a function of the potential.  They show instead that the mean DN map need not lie in the nonlinear range in an identifiable way.  For deterministic potentials, a closely related concavity statement for Schr\"odinger DN maps on compact Riemannian manifolds near the zero potential appears in \cite{CarsteaFeizmohammadiOksanen2023}.  The proposition below is the corresponding Jensen-type averaged form in the present Euclidean setting.

\begin{proposition}\label{prop:DN-Jensen-potential}
Let $q$ be a strongly measurable $\mathcal Q_+^\infty(\OmegaD)$-valued random potential with
\[
  \|q(\omega,\cdot)\|_{L^\infty(\OmegaD)}\le C
  \quad\text{for almost every }\omega,
\]
and put $\overline q=\E q$.  Then
\begin{equation}\label{eq:DN-Jensen-potential}
  \E\Lambda_q\le \Lambda_{\overline q}
\end{equation}
as quadratic forms on $H^{1/2}(\partial\OmegaD)$.  If equality holds and $\OmegaD$ is connected, then $q=\overline q$ almost surely.
\end{proposition}

\begin{proof}
For a boundary value $f$, let $u_{q,f}$ minimize
\[
  \int_\OmegaD (|\nabla u|^2+q|u|^2)\,dx
\]
among functions with trace $f$.  Then
\[
  \la\Lambda_qf,f\ra
  \le \int_\OmegaD (|\nabla u_{\overline q,f}|^2+q|u_{\overline q,f}|^2)\,dx.
\]
Taking expectations gives \eqref{eq:DN-Jensen-potential}.

If equality holds as a quadratic form, then for $f\equiv1$ the energy defect vanishes almost surely:
\[
  \int_\OmegaD (|\nabla(u_{q,1}-u_{\overline q,1})|^2
  +q|u_{q,1}-u_{\overline q,1}|^2)\,dx=0.
\]
Thus $u_{q,1}=u_{\overline q,1}$ almost surely.  Subtracting the two equations gives
\[
  (q-\overline q)u_{\overline q,1}=0.
\]
Since $\overline q\ge0$ and the boundary value is $1$, the maximum principle gives $u_{\overline q,1}>0$ in $\OmegaD$.  Hence $q=\overline q$ almost surely.
\end{proof}

\section{The full law of the Schr\"odinger DN map}
\label{sec:law-DN-map}

Expectations and finite moment hierarchies may lose information.  The full law of the DN map does not, provided the corresponding deterministic inverse problem is injective.  We formulate this in a countable way in order to avoid putting an unnecessary topology on a space of boundary operators.

Let
\[
  H_\R=H^{1/2}(\partial\OmegaD;\R)
\]
and fix a countable dense set $\{f_j\}_{j=1}^\infty\subset H_\R$.  For a coefficient $q$ for which the DN map is well defined, set
\[
  T(q)=\big(\la\Lambda_q f_i,f_j\ra\big)_{i,j\ge1}
  \in\R^{\mathbb N\times\mathbb N}.
\]
In this section, the full law of the DN map means the law of this countable matrix representation.  This convention loses no information about the DN map.

\begin{lemma}\label{lem:countable-DN-representation}
Let $B_1$ and $B_2$ be bounded bilinear forms on $H_\R$.  If
\[
  B_1(f_i,f_j)=B_2(f_i,f_j),\qquad i,j\ge1,
\]
then $B_1=B_2$ on $H_\R\times H_\R$.  Consequently, for real Schr\"odinger potentials,
\[
  T(q_1)=T(q_2)
  \quad\Longleftrightarrow\quad
  \Lambda_{q_1}=\Lambda_{q_2}
\]
as real DN maps.
\end{lemma}

\begin{proof}
Let $f,g\in H_\R$.  Choose sequences $f_{i_m}\to f$ and $f_{j_m}\to g$ in $H_\R$.  The equality on the dense set and the continuity of $B_1$ and $B_2$ give
\[
  B_1(f,g)=\lim_{m\to\infty}B_1(f_{i_m},f_{j_m})
  =\lim_{m\to\infty}B_2(f_{i_m},f_{j_m})=B_2(f,g).
\]
The assertion for DN maps follows by applying this to the associated boundary bilinear forms.  The complex sesquilinear DN map is then recovered by complexification, or equivalently by the usual polarization identities.
\end{proof}

\begin{proposition}\label{prop:injective-observable-law-coeff}
Let $X$ and $Y$ be standard Borel spaces and let $F:X\to Y$ be an injective Borel map.  If $x$ and $\widetilde x$ are $X$-valued random variables such that
\[
  \Law(F(x))=\Law(F(\widetilde x)),
\]
then
\[
  \Law(x)=\Law(\widetilde x).
\]
\end{proposition}

\begin{proof}
By the Lusin--Souslin theorem, $F(X)$ is a Borel subset of $Y$ and the inverse map $F^{-1}:F(X)\to X$ is Borel \cite[Theorem~15.1]{Kechris1995}.  Hence
\[
  \Law(x)=(F^{-1})_\#\Law(F(x))
  =(F^{-1})_\#\Law(F(\widetilde x))=\Law(\widetilde x).
\]
\end{proof}

\begin{corollary}\label{cor:law-DN-determines-law-q}
Let $\mathcal Q$ be a standard Borel coefficient class on which the Schr\"odinger DN map is defined.  Assume that $T:\mathcal Q\to\R^{\mathbb N\times\mathbb N}$ is Borel and that deterministic uniqueness holds on $\mathcal Q$, namely
\[
  \Lambda_{q_1}=\Lambda_{q_2}
  \quad\Longrightarrow\quad
  q_1=q_2,
  \qquad q_1,q_2\in\mathcal Q.
\]
If $q$ and $\widetilde q$ are $\mathcal Q$-valued random variables and
\[
  \Law(T(q))=\Law(T(\widetilde q)),
\]
then
\[
  \Law(q)=\Law(\widetilde q).
\]
In particular, the full law of the DN map determines the law of the random coefficient on any deterministic uniqueness class for which the countable coordinate map $T$ is Borel.  For the smooth nonnegative class $\mathcal Q_+^\infty(\OmegaD)$ with its usual Fr\'echet Borel structure, this Borel property follows from the continuous dependence of the Dirichlet problem on $q$ for fixed boundary values.  The analogous statement holds for conductivities on any class on which the corresponding deterministic DN map is uniquely solvable and the associated countable coordinate map is Borel.
\end{corollary}

\begin{proof}
By Lemma~\ref{lem:countable-DN-representation}, deterministic uniqueness implies that $T$ is injective on $\mathcal Q$.  Apply Proposition~\ref{prop:injective-observable-law-coeff} with $X=\mathcal Q$, $Y=\R^{\mathbb N\times\mathbb N}$, and $F=T$.
\end{proof}

\begin{remark}
This formal result separates two questions.  The full law of the DN map, understood as the joint law of all countably many DN matrix entries above, is as strong as deterministic uniqueness.  Low-order DN moments are not.
\end{remark}

\section{Averaged Green's operators for random potentials}
\label{sec:Green-potential}

We next consider averaged interior Green's operators.  They behave differently because the interior symbol of $G_q$ contains local information about the coefficient.  All symbolic statements in this section are local in the interior of $\OmegaD$.  More precisely, if $\chi,\psi\in C_c^\infty(\OmegaD)$ and $\psi=1$ near $\supp\chi$, then the symbol of $G_q$ at points of $\supp\chi$ means the full symbol of the localized operator $\chi G_q\psi$.  This convention removes boundary contributions.  In what follows, $\sigma_{-j}$ denotes the homogeneous component of degree $-j$ of this interior symbol.

Let $(\mathcal X,\mathcal F,\Prob)$ be a probability space and let
\[
  q:\mathcal X\to \mathcal Q_+^\infty(\OmegaD)
\]
be a strongly measurable random potential.  We assume throughout this section that, for every $N\ge0$, there exists $C_N$ with
\[
  \|q(\omega,\cdot)\|_{C^N(\overline\OmegaD)}\le C_N
\]
for almost every $\omega$.  Put
\[
  \overline q=\E q,
  \qquad
  \overline G=\E G_q.
\]
The uniform $C^N$ bounds imply that $\overline q\in C^\infty(\overline\OmegaD)$ and that differentiation commutes with expectation for all derivatives used in the local symbol calculus.

\begin{proposition}\label{prop:averaged-G-well-defined}
The expectation $\overline G=\E G_q$ is a well-defined bounded positive self-adjoint operator on $L^2(\OmegaD)$, characterized by
\[
  \la\overline Gf,g\ra=\E\la G_qf,g\ra.
\]
For any $\chi,\psi\in C_c^\infty(\OmegaD)$ with $\psi=1$ near $\supp\chi$, the localized operator $\chi\overline G\psi$ is a classical pseudodifferential operator of order $-2$ in the interior, and its full symbol is obtained by averaging the corresponding full interior symbols of $\chi G_{q(\omega)}\psi$.
\end{proposition}

\begin{proof}
Since $q\ge0$, the operators $A_q=-\Delta_D+q$ satisfy $A_q\ge\lambda_1 I$, where $\lambda_1>0$ is the first Dirichlet eigenvalue of $-\Delta$.  Hence $\|G_q\|_{L^2\to L^2}\le\lambda_1^{-1}$ uniformly.  The resolvent identity gives
\[
  G_q-G_{\widetilde q}=-G_q(q-\widetilde q)G_{\widetilde q},
\]
and hence $q\mapsto G_q$ is continuous from bounded subsets of $L^\infty(\OmegaD)$, with the $L^\infty$ norm, into $\mathcal L(L^2(\OmegaD))$.  The strong measurability of $q$ therefore implies weak measurability of $\omega\mapsto G_{q(\omega)}$.  Thus the sesquilinear form $(f,g)\mapsto\E\la G_qf,g\ra$ is bounded, positive, and symmetric on $L^2(\OmegaD)$.  By the Riesz representation theorem it defines a bounded positive self-adjoint operator, which is the weak operator expectation of $G_q$.

For the symbolic assertion, fix $\chi,\psi\in C_c^\infty(\OmegaD)$ with $\psi=1$ near $\supp\chi$.  The standard interior parametrix construction for $-\Delta+q(\omega)$ gives full classical symbols for $\chi G_{q(\omega)}\psi$ with seminorms controlled by finitely many $C^N$ norms of $q(\omega)$.  The assumed uniform bounds allow each homogeneous symbol coefficient to be averaged term by term.  This gives the claimed interior symbol for $\chi\overline G\psi$.
\end{proof}

\begin{proposition}\label{prop:Green-Jensen}
With the notation above,
\begin{equation}\label{eq:Green-Jensen}
  G_{\overline q}\le \overline G=\E G_q
\end{equation}
as quadratic forms on $L^2(\OmegaD)$.  Equality holds if and only if $q=\overline q$ almost surely.
\end{proposition}

\begin{proof}
Write $K=q-\overline q$.  Expanding the second resolvent identity to second order around $A_{\overline q}$ gives
\begin{equation}\label{eq:second-order-resolvent}
  G_q=G_{\overline q}-G_{\overline q}KG_{\overline q}
       +G_{\overline q}KG_qKG_{\overline q}.
\end{equation}
After taking expectations, the first-order term vanishes and
\[
  \overline G-G_{\overline q}
  =\E\left[G_{\overline q}KG_qKG_{\overline q}\right].
\]
Thus, for $f\in L^2(\OmegaD)$,
\[
  \la(\overline G-G_{\overline q})f,f\ra
  =\E\la G_qKG_{\overline q}f,KG_{\overline q}f\ra\ge0.
\]
If equality holds for all $f$, then $KG_{\overline q}f=0$ almost surely for each fixed $f$, since $G_q$ is positive and injective.  Choose a countable $L^2$-dense set $\{\phi_m\}_{m=1}^\infty\subset C_c^\infty(\OmegaD)$ and put $f_m=A_{\overline q}\phi_m$.  Intersecting the corresponding full-measure sets gives $K\phi_m=0$ for all $m$, almost surely.  Since multiplication by $K$ is a bounded operator on $L^2(\OmegaD)$ and $\{\phi_m\}$ is dense, $K=0$ in $L^\infty(\OmegaD)$ almost surely.  The converse is immediate.
\end{proof}

\begin{proposition}\label{prop:Green-mean-variance}
The full interior symbol of the averaged Green's operator $\overline G=\E G_q$ determines  the  functions
\[
  \E q
  \qquad\text{and}\qquad
  \Var(q)=\E[(q-\E q)^2]
\]
pointwise in $\OmegaD$.  More precisely, if $G_0=(-\Delta_D)^{-1}$, then in the interior symbol calculus,
\begin{equation}\label{eq:recover-mean-symbol}
  \overline q(x)=-|\xi|^4\sigma_{-4}(\overline G-G_0)(x,\xi),
  \qquad \xi\ne0,
\end{equation}
and
\begin{equation}\label{eq:recover-var-symbol}
  \Var(q)(x)=|\xi|^6\sigma_{-6}(\overline G-G_{\overline q})(x,\xi),
  \qquad \xi\ne0.
\end{equation}
\end{proposition}

\begin{proof}
The localized interior parametrix for $-\Delta+q$ gives
\[
  \sigma(G_q)(x,\xi)=|\xi|^{-2}-q(x)|\xi|^{-4}+O(|\xi|^{-5}).
\]
The coefficient of order $-4$ is local and is unaffected by the boundary because of the interior localization convention above.  Averaging therefore gives
\[
  \sigma_{-4}(\overline G-G_0)(x,\xi)=-\overline q(x)|\xi|^{-4},
\]
which is \eqref{eq:recover-mean-symbol}.

Having recovered $\overline q$, the deterministic operator $G_{\overline q}$ and its interior symbol are known.  Put $K=q-\overline q$.  The second resolvent identity \eqref{eq:second-order-resolvent} gives
\[
  \overline G-G_{\overline q}
  =\E\bigl[G_{\overline q}KG_qKG_{\overline q}\bigr].
\]
The factors $G_{\overline q}$, $G_q$, and $G_{\overline q}$ have principal symbols $|\xi|^{-2}$, while multiplication by $K$ has symbol $K(x)$.  Hence the leading symbol of the product has order $-6$ and equals
\[
  K(x)^2|\xi|^{-6}.
\]
All other symbolic composition terms are of lower homogeneous order.  Averaging gives
\[
  \sigma_{-6}(\overline G-G_{\overline q})(x,\xi)
  =\E[K(x)^2]|\xi|^{-6},
\]
which is \eqref{eq:recover-var-symbol}.
\end{proof}

\begin{lemma}\label{lem:Green-symbol-triangular}
For each $k\ge1$, the homogeneous component of order $-2k-2$ in the interior symbol of $G_q=(-\Delta_D+q)^{-1}$ has the form
\begin{equation}\label{eq:triangular-Green-symbol}
  \sigma_{-2k-2}(G_q)(x,\xi)
  =(-1)^kq(x)^k|\xi|^{-2k-2}
   +P_k\bigl(x,\xi;\{\partial^\alpha q(x):|\alpha|\le 2k-2\}\bigr),
\end{equation}
where $P_k$ is a universal polynomial in the displayed derivatives of $q$ and rational homogeneous functions of $\xi$, and every monomial in $P_k$ contains at most $k-1$ factors involving derivatives of $q$.  The undifferentiated function $q$ itself counts as one such factor.
\end{lemma}

\begin{proof}
The assertion is local, so we work after inserting cutoffs
$\chi,\psi\in C_c^\infty(\OmegaD)$ with $\psi=1$ near $\supp\chi$.
Let $G_0=(-\Delta_D)^{-1}$.  The resolvent identity
\[
  G_q=G_0-G_0qG_q
\]
implies, for each $N\ge0$,
\begin{equation}\label{eq:finite-resolvent-expansion}
  G_q=\sum_{j=0}^N(-1)^jG_0(qG_0)^j
       +(-1)^{N+1}G_0(qG_0)^NqG_q .
\end{equation}
After localization, $G_0$ and $G_q$ are pseudodifferential operators of order $-2$, multiplication by $q$ has order zero, and therefore the remainder in \eqref{eq:finite-resolvent-expansion} has order $-2N-4$.  Taking $N=k$, the homogeneous symbol of order $-2k-2$ of $G_q$ is determined only by the finitely many terms
\[
  G_0,\, G_0qG_0,\,\dots,\, G_0(qG_0)^k .
\]

For the $j$th term, $G_0(qG_0)^j$ has order $-2j-2$.  Its leading homogeneous symbol is
\[
  q(x)^j|\xi|^{-2j-2},
\]
because the localized free Green's operator has leading interior symbol $|\xi|^{-2}$ and the top term in a symbolic composition is the product of top terms.  Lower homogeneous components of this same $j$th term are obtained from the Kohn--Nirenberg composition formula.  If $A$ and $B$ have full symbols $a$ and $b$, then, with $D_x=(1/i)\partial_x$,
\[
  \sigma(A\circ B)(x,\xi)
  \sim
  \sum_{\alpha\in\mathbb N^n}\frac1{\alpha!}
  \partial_\xi^\alpha a(x,\xi)\,D_x^\alpha b(x,\xi).
\]
Each monomial obtained from this formula still contains at most $j$ factors coming from the $j$ multiplications by $q$, with derivatives allowed to fall on those factors.  Moreover, to contribute to order $-2k-2$ with $j<k$, the total loss of symbolic order is $2(k-j)$, so no derivative of $q$ of order larger than $2(k-j)\le 2k-2$ can occur.

Thus the term with $j=k$ contributes at order $-2k-2$ only through its leading symbol
\[
  (-1)^kq(x)^k|\xi|^{-2k-2}.
\]
All other contributions at the same order come from indices $j<k$ and contain at most $j\le k-1$ factors of derivatives of $q$.  Collecting these lower-degree contributions gives the universal polynomial $P_k$ in \eqref{eq:triangular-Green-symbol}.
\end{proof}

\subsection{Two-atom averaged Green's operators}

The two-atom case has more algebraic rigidity than a general random field. We start with the following lemma.
\begin{lemma}\label{lem:two-atom-jet-closure}
Let $q_1,q_2\in C^\infty$ on an open set and let $0<\tau<1$ be fixed.  Put
\[
  m_j=\tau q_1^j+(1-\tau)q_2^j.
\]
For each $d\ge1$, the derivatives of $m_1,\dots,m_d$ determine every finite linear combination of expressions of the form
\[
  \tau\prod_{\nu=1}^r \partial^{\alpha_\nu}q_1
  +(1-\tau)\prod_{\nu=1}^r \partial^{\alpha_\nu}q_2,
  \qquad 0\le r\le d,
\]
where the multi-indices $\alpha_\nu$ are arbitrary.  Thus a differentiated occurrence and an undifferentiated occurrence of $q_\ell$ are both counted as one factor.
\end{lemma}

\begin{proof}
The assertion is local.  Write $s=1-\tau$ and set
\[
  \mu=m_1,
  \qquad
  V=m_2-m_1^2=\tau s(q_1-q_2)^2.
\]
Let $U_0$ be the interior of $\{V=0\}$ and let $U_1=\{V>0\}$.  On $U_0$ one has $q_1=q_2=\mu$, so
\[
  \tau\prod_{\nu=1}^r \partial^{\alpha_\nu}q_1
  +s\prod_{\nu=1}^r \partial^{\alpha_\nu}q_2
  =\prod_{\nu=1}^r \partial^{\alpha_\nu}\mu.
\]
Thus the claim is immediate on $U_0$.

We next work on $U_1$.  The degree-one expressions are exactly the derivatives of $\mu$.  For degree two, let $\delta=q_1-q_2$.  Then
\[
  \delta^2=\frac{V}{\tau s},
  \qquad
  q_1=\mu+s\delta,
  \qquad
  q_2=\mu-\tau\delta .
\]
The sign of $\delta$ is not determined by $\mu$ and $V$ alone.  However, every weighted expression of degree at most two is independent of this sign.  For example, for arbitrary multi-indices $\alpha,\beta$,
\begin{equation}\label{eq:two-atom-quadratic-jets}
\begin{split}
&\tau(\partial^\alpha q_1)(\partial^\beta q_1)
+s(\partial^\alpha q_2)(\partial^\beta q_2) \\
&\hspace{1cm}
=(\partial^\alpha\mu)(\partial^\beta\mu)
 +\tau s(\partial^\alpha\delta)(\partial^\beta\delta),
\end{split}
\end{equation}
and the last product is determined by $V$ because both derivatives of $\delta$ change sign when the branch is changed.  Equivalently, it can be written locally in terms of $V$ and its derivatives on $U_1$.  The special case $|\alpha|=|\beta|=1$ gives the explicit formula
\[
\begin{split}
&\tau(\partial_iq_1)(\partial_jq_1)
 +s(\partial_iq_2)(\partial_jq_2) \\
&\hspace{1cm}=\partial_i\mu\,\partial_j\mu
  +\frac{\partial_iV\,\partial_jV}{4V}.
\end{split}
\]
Terms containing one undifferentiated factor of $q_\ell$ and one differentiated factor are obtained in the same way, or by differentiating $m_2=\tau q_1^2+s q_2^2$ and subtracting already determined quadratic derivative products.  Hence all weighted expressions with at most two factors are determined.

It remains to handle degrees $d\ge3$.  If $\tau\ne1/2$, the centered third moment
\[
  c_3=m_3-3m_1m_2+2m_1^3
  =\tau s(1-2\tau)(q_1-q_2)^3
\]
determines the signed difference on $U_1$, since
\[
  q_1-q_2=\frac{c_3}{(1-2\tau)V}.
\]
Consequently
\[
  q_1=\mu+s\frac{c_3}{(1-2\tau)V},
  \qquad
  q_2=\mu-\tau\frac{c_3}{(1-2\tau)V},
\]
and differentiating these formulas determines the full jets of both labeled atoms from the derivatives of $m_1,m_2,m_3$.  Thus any weighted expression with at most $d$ factors involving derivatives of the atoms is determined from the derivatives of $m_1,\dots,m_d$.

If $\tau=1/2$, then the labels carry equal weights.  On each component of $U_1$ the functions $m_1$ and $m_2$ determine the unordered pair of atoms, equivalently the two branches $\mu\pm\sqrt V$.  A local choice of one branch determines both jets, and changing the branch merely interchanges the two equal-weight summands.  Hence the weighted expression is branch-independent and is determined by the derivatives of $m_1$ and $m_2$.

Finally, $U_0\cup U_1$ is dense.  Its complement is contained in the boundary of the closed set $\{V=0\}$, and every such boundary point is a limit point of $U_1$.  The original weighted expression is smooth.  Therefore its values on the remaining boundary points are uniquely fixed by continuous extension from the already determined values on $U_0\cup U_1$.
\end{proof}

We then have

\begin{proposition}\label{prop:two-atom-Green-moments}
Let $q_1,q_2\in \mathcal Q_+^\infty(\OmegaD)$ and let
\[
  E=\tau G_{q_1}+(1-\tau)G_{q_2},
  \qquad 0<\tau<1.
\]
Then the full interior symbol of $E$ determines
\[
  m_k=\tau q_1^k+(1-\tau)q_2^k,
  \qquad k=1,2,3,\ldots.
\]
In particular, $E$ determines the pointwise two-atom measure
\[
  \tau\delta_{q_1(x)}+(1-\tau)\delta_{q_2(x)}
\]
for each $x\in\OmegaD$.
\end{proposition}

\begin{proof}
The assertion is local.  The order $-4$ symbol gives $m_1$.  Suppose that $m_1,\dots,m_{k-1}$ have been recovered.  By Lemma~\ref{lem:Green-symbol-triangular}, the order $-2k-2$ symbol of $G_q$ consists of the leading term $(-1)^kq^k|\xi|^{-2k-2}$ plus lower terms containing at most $k-1$ factors of derivatives of $q$.  After taking the weighted two-atom average, Lemma~\ref{lem:two-atom-jet-closure} determines the weighted averages of all these lower terms from the derivatives of $m_1,\dots,m_{k-1}$.  Removing them from the known symbol of $E$ leaves
\[
  (-1)^k m_k(x)|\xi|^{-2k-2}.
\]
Thus $m_k$ is recovered recursively.

The pointwise two-atom measure is determined by its full moment sequence.  There is no moment-determinacy issue here, since the measure has finite support.  More concretely, write
\[
  V=m_2-m_1^2=\tau(1-\tau)(q_1-q_2)^2.
\]
If $V(x)=0$, then both atoms coincide and the measure is $\delta_{m_1(x)}$.  If $V(x)>0$ and $\tau=1/2$, the atoms are
\[
  m_1(x)\pm \sqrt{V(x)}.
\]
If $V(x)>0$ and $\tau\ne1/2$, then with
\[
  c_3=m_3-3m_1m_2+2m_1^3
\]
one has
\[
  q_1-q_2=\frac{c_3}{(1-2\tau)V},
  \qquad
  q_1=m_1+(1-\tau)(q_1-q_2),
  \qquad
  q_2=m_1-\tau(q_1-q_2).
\]
Thus the weighted two-atom measure is determined pointwise.
\end{proof}

\begin{corollary}\label{cor:Green-rigidity-potential}
If $q,q_1,q_2\in \mathcal Q_+^\infty(\OmegaD)$ and
\[
  G_q=\tau G_{q_1}+(1-\tau)G_{q_2},
  \qquad 0<\tau<1,
\]
then $q=q_1=q_2$ in $\OmegaD$.
\end{corollary}

\begin{proof}
The order $-4$ symbol gives $q=\tau q_1+(1-\tau)q_2$.  At order $-6$, the lower symbolic terms are linear in the jet of the potential, by Lemma~\ref{lem:Green-symbol-triangular} with $k=2$.  Their weighted average for $q_1,q_2$ therefore agrees with the corresponding lower term for $q$.  After this cancellation, the order $-6$ symbols give
\[
  0=\tau(q_1-q)^2+(1-\tau)(q_2-q)^2
  =\tau(1-\tau)(q_1-q_2)^2.
\]
Thus $q_1=q_2=q$.
\end{proof}

\begin{remark}[Why general randomness is different]\label{rem:general-random-obstruction}
For a general random field, $\E G_q$ determines a hierarchy of averaged local jet invariants, not only pointwise moments $\E q^k$.  This is already visible in one dimension.  In the Kohn--Nirenberg convention,
\begin{align*}
 \sigma(( -\partial_x^2+q)^{-1})(x,\xi)
 & =\xi^{-2}-q\xi^{-4}-2iq'\xi^{-5}
 +(q^2+3q'')\xi^{-6} \\
 &\quad +i(6qq'+4q''')\xi^{-7} \\
 &\quad -(q^3+13qq''+10(q')^2+5q'''')\xi^{-8}+\cdots.
\end{align*}
After averaging, the coefficient of $\xi^{-8}$ contains not only $\E q^3$, but also $\E(qq'')$ and $\E((q')^2)$.  Thus higher pointwise moments are mixed with derivative correlations.  The two-atom case has algebraic closure relations that make these derivative correlations dependent on lower moments.  For a general random field there is no analogous closure relation: the averaged symbol records mixed jet correlations, not only the pointwise moments $\E q(x)^k$.
\end{remark}

\subsection{Finite-dimensional random potential models}

The preceding obstruction comes from the freedom of the random field.  If the randomness is finite-dimensional with known spatial profiles, the averaged Green's operator can again determine the law.

\begin{proposition}\label{prop:finite-dimensional-Green-model}
Let
\[
  q_X(x)=q_0(x)+\sum_{j=1}^d X_jV_j(x),
\]
where $q_0,V_1,\dots,V_d\in C^\infty(\overline\OmegaD)$ are known, and where the law of $X=(X_1,\dots,X_d)$ is compactly supported and satisfies $q_X\ge0$ almost surely.  Assume that there is a nonempty open set $U\Subset\OmegaD$ such that, for every $k\ge1$, the functions
\[
  V^\alpha=V_1^{\alpha_1}\cdots V_d^{\alpha_d},
  \qquad |\alpha|=k,
\]
are linearly independent on $U$.  Then the averaged Green's operator
\[
  \E G_{q_X}
\]
determines all mixed moments $\E X^\alpha$.  Consequently it determines the law of $X$.
\end{proposition}

\begin{proof}
We first recover the functions
\[
  M_k(x)=\E[q_X(x)^k],
  \qquad k=1,2,3,\ldots.
\]
The recovery is recursive.  By Lemma~\ref{lem:Green-symbol-triangular}, at the symbolic order where $q_X^k$ first appears, all lower-order derivative contributions involve at most $k-1$ factors of $q_X$ or its derivatives.  Since every derivative of $q_X$ is affine in $X$, these lower contributions involve only mixed moments of degree $<k$, already known by induction.  Hence the known symbol of $\E G_{q_X}$ determines $M_k$ after the lower-degree terms have been subtracted.

Assume the mixed moments of degree $<k$ have been recovered.  Expanding $M_k$ gives
\[
  M_k(x)=\sum_{\ell=0}^k\binom{k}{\ell}q_0(x)^{k-\ell}
  \sum_{|\alpha|=\ell}\binom{\ell}{\alpha}
  \E[X^\alpha]V(x)^\alpha.
\]
All terms with $\ell<k$ are known.  Subtracting them leaves
\[
  \sum_{|\alpha|=k}\binom{k}{\alpha}\E[X^\alpha]V^\alpha(x).
\]
Restricting to $U$ and using the stated linear independence of the functions $V^\alpha$, the coefficients $\E[X^\alpha]$ with $|\alpha|=k$ are determined.  This proves by induction that all mixed moments of $X$ are determined.  Since the law of $X$ is compactly supported in $\mathbb R^d$, its moment sequence determines its law; for example, polynomials are dense in the continuous functions on any compact set containing the support.
\end{proof}

\begin{remark}
This finite-dimensional result is included as a model case, not as a replacement for the general random-field question.  It shows how the obstruction in Remark~\ref{rem:general-random-obstruction} disappears when all jet correlations come from the same finite random vector.
\end{remark}

\appendix

\section{Conductivity boundary counterexamples}
\label{app:conductivity-DN-counterexamples}

The results in this appendix are the conductivity analogues of the boundary-map obstructions in the main text.  The deterministic conductivity class used here is
\[
  \mathcal G_+^\infty(\OmegaD)
  =\{\gamma\in C^\infty(\overline\OmegaD;\R):\gamma>0\text{ on }\overline\OmegaD\}.
\]
All random conductivities in the appendix take values in this class, with the uniform lower, upper, and differentiability bounds stated where averaged Green's symbols are used.  We first recall the deterministic notation used for the conductivity equation.

For $\gamma\in\mathcal G_+^\infty(\OmegaD)$, let
\[
  R_\gamma=(-\nabla\cdot\gamma\nabla)^{-1}_{D}
\]
be the Dirichlet conductivity Green's operator, and let
\[
  \Lambda_\gamma f=\gamma\partial_\nu u_f|_{\partial\OmegaD},
  \qquad
  \nabla\cdot(\gamma\nabla u_f)=0,
  \qquad
  u_f|_{\partial\OmegaD}=f.
\]
Then
\begin{equation}\label{eq:Alessandrini-gamma}
  \la(\Lambda_{\gamma_1}-\Lambda_{\gamma_2})f,g\ra
  =\int_\OmegaD (\gamma_1-\gamma_2)
  \nabla u_f^{(1)}\cdot\nabla u_g^{(2)}\,dx.
\end{equation}

\subsection{A radial conductivity counterexample}

The conductivity counterexample follows from the preceding one by the Liouville transform.  In this subsection we work in dimension three and write $\mu_{\ell,N}=\mu_{\ell,N}^{(3)}$.  Let $Q_N$ be the family \eqref{eq:QN-family}, and let $w_N$ be the positive radial solution of
\[
  (-\Delta+Q_N)w_N=0\quad\text{in }B(0,1),
  \qquad w_N|_{\partial B(0,1)}=1.
\]
Then $Q_N=\Delta w_N/w_N$.  If $\gamma_N=w_N^2$, then $u$ solves
\[
  \nabla\cdot(\gamma_N\nabla u)=0
\]
if and only if $v=w_Nu$ solves
\[
  (-\Delta+Q_N)v=0.
\]
Since $w_N=1$ on $\partial B(0,1)$,
\[
  \partial_\nu(w_Nu)=\partial_\nu u+f\partial_\nu w_N
  \quad\text{on }\partial B(0,1),
\]
and therefore
\[
  \Lambda_{\gamma_N}f=\Lambda_{Q_N}f-f\partial_\nu w_N.
\]
As $w_N$ is radial, $\partial_\nu w_N$ is constant on the boundary.  Moreover $w_N$ is the Schr\"odinger solution with boundary value $1$, so $\partial_\nu w_N=\mu_{0,N}$.  Hence the conductivity DN eigenvalues are
\begin{equation}\label{eq:kappa-from-mu}
  \kappa_{\ell,N}=\mu_{\ell,N}-\mu_{0,N}.
\end{equation}
Using \eqref{eq:mu-N-formula-general},
\begin{equation}\label{eq:kappa-N-formula}
  \kappa_{\ell,N}
  =\ell+\frac{\lambda^2-4N^2}{2}
  \left(\frac1{\lambda+2\ell+1}-\frac1{\lambda+1}\right).
\end{equation}
This is affine in $N^2$.

\begin{proposition}\label{prop:conductivity-DN-counterexample}
Let $\OmegaD=B(0,1)\subset\R^3$.  Define
\begin{align*}
 w_1(r)&=\frac{27+7r^2}{17(9-7r^2)},\\[3pt]
 w(r)&=\frac{143+27r^{10}}{17(13-3r^{10})},\\[3pt]
 w_2(r)&=\frac{225+13r^{14}}{17(15-r^{14})},
\end{align*}
and set
\[
  \gamma_1=w_1^2,
  \qquad
  \gamma=w^2,
  \qquad
  \gamma_2=w_2^2.
\]
Then $\gamma,\gamma_1,\gamma_2$ are positive real analytic radial conductivities in a neighborhood of $\overline{B(0,1)}$, they are pairwise distinct, and
\begin{equation}\label{eq:conductivity-half-counterexample}
  \Lambda_\gamma=\frac12\Lambda_{\gamma_1}+\frac12\Lambda_{\gamma_2}.
\end{equation}
\end{proposition}

\begin{proof}
The displayed functions are the corresponding positive solutions $w_N$ for $\lambda=16$ and $N=1,5,7$.  More generally, the regular radial solution for the zero spherical harmonic in Lemma~\ref{lem:radial-affine-family}, normalized by $w_N(1)=1$, is
\[
  w_N(r)=\frac{1-d_N}{1-\beta_N}
  \frac{1-\beta_Nr^{2N}}{1-d_Nr^{2N}},
  \qquad
  \beta_N=d_N\frac{1-2N}{1+2N}.
\]
For the three values $N=1,5,7$ the denominators are nonzero near $\overline B(0,1)$, and the displayed $w_N$ are positive there.  Thus $\gamma_N=w_N^2$ are positive real analytic conductivities.  Formula \eqref{eq:kappa-N-formula} and the identity
\[
  5^2=\frac12\cdot1^2+\frac12\cdot7^2
\]
show that
\[
  \kappa_{\ell,5}=\frac12\kappa_{\ell,1}+\frac12\kappa_{\ell,7}
  \qquad\text{for every }\ell\ge0.
\]
Diagonalization in spherical harmonics gives \eqref{eq:conductivity-half-counterexample}.  The three conductivities are pairwise distinct, for instance because
\[
  \gamma_1(0)=\frac9{289},\qquad
  \gamma(0)=\frac{121}{289},\qquad
  \gamma_2(0)=\frac{225}{289}.
\]
\end{proof}

\begin{corollary}\label{cor:random-conductivity-mean-DN-deterministic}
There is a two-valued random conductivity $\Gamma$ and a deterministic conductivity $\gamma_*$ such that
\[
  \E[\Lambda_\Gamma]=\Lambda_{\gamma_*},
\]
while
\[
  \gamma_*\ne \E\Gamma.
\]
\end{corollary}

\begin{proof}
Let $\Gamma=\gamma_1$ and $\Gamma=\gamma_2$ with probabilities $1/2$ and $1/2$, where $\gamma_1,\gamma_2$ are the conductivities in Proposition~\ref{prop:conductivity-DN-counterexample}, and take $\gamma_*=\gamma$.  Then \eqref{eq:conductivity-half-counterexample} gives $\E\Lambda_\Gamma=\Lambda_\gamma$.  On the other hand,
\[
  \gamma(0)=\frac{121}{289},
  \qquad
  \frac12(\gamma_1(0)+\gamma_2(0))=\frac{117}{289}.
\]
Thus $\gamma\ne\E\Gamma$.
\end{proof}

\begin{remark}
The same barycentric construction gives conductivity DN identities in every dimension $n\ge2$, after replacing $\gamma_N$ by $w_N^2$, where $w_N$ is the positive radial solution of $(-\Delta+Q_N)w_N=0$ with $w_N|_{\partial B}=1$.  The explicit rational formulas above are the three-dimensional instances used in Proposition~\ref{prop:conductivity-DN-counterexample}.
\end{remark}

\begin{remark}
The conductivities in Proposition~\ref{prop:conductivity-DN-counterexample} are not scalar multiples.  Indeed,
\[
  \gamma_1(1)=\gamma(1)=\gamma_2(1)=1,
\]
but
\[
  \gamma_1(0)=\frac9{289},
  \qquad
  \gamma(0)=\frac{121}{289},
  \qquad
  \gamma_2(0)=\frac{225}{289}.
\]
Thus the example is not a disguised instance of the trivial scaling relation $\Lambda_{\alpha\gamma}=\alpha\Lambda_\gamma$.
\end{remark}

\section{Conductivity Green's operators}
\label{sec:conductivity-Green}

We record the corresponding symbolic facts for conductivity Green's operators.  Principal-coefficient randomness is organized differently from lower-order potential randomness.  As in the Schr\"odinger Green's operator section, all symbol statements are local in the interior: if $\chi,\psi\in C_c^\infty(\OmegaD)$ and $\psi=1$ near $\supp\chi$, the symbol of $R_\gamma$ at points of $\supp\chi$ means the full symbol of the localized operator $\chi R_\gamma\psi$.

Let $\gamma$ be a strongly measurable $\mathcal G_+^\infty(\OmegaD)$-valued random conductivity, uniformly bounded above and below and uniformly bounded in $C^N(\overline\OmegaD)$ for every $N$.  Set
\[
  a(\omega,x)=\gamma(\omega,x)^{-1},
  \qquad
  m(x)=\E a(\cdot,x),
  \qquad
  \gamma_h=m^{-1}.
\]
Let
\[
  \overline R=\E R_\gamma.
\]

\begin{proposition}\label{prop:conductivity-principal-symbol}
In the above localized interior sense, $\overline R$ is a classical pseudodifferential operator of order $-2$, and
\[
  \sigma_{-2}(\overline R)(x,\xi)=m(x)|\xi|^{-2},
  \qquad
  \sigma_{-3}(\overline R)(x,\xi)=i\nabla m(x)\cdot\xi\,|\xi|^{-4}.
\]
Thus $\overline R$ determines $m=\E(\gamma^{-1})$.
\end{proposition}

\begin{proof}
For a deterministic conductivity, write $a=\gamma^{-1}$.  The symbol of $-\nabla\cdot\gamma\nabla$ is
\[
  p_2=a^{-1}|\xi|^2,
  \qquad
  p_1=ia^{-2}\nabla a\cdot\xi.
\]
The parametrix recursion gives
\[
  \sigma_{-2}(R_\gamma)=a|\xi|^{-2},
  \qquad
  \sigma_{-3}(R_\gamma)=i\nabla a\cdot\xi\,|\xi|^{-4}.
\]
Averaging the uniformly controlled localized interior symbols gives the claim.  The weak expectation is justified as in the Schr\"odinger case, using the continuity of the Dirichlet resolvent under uniformly elliptic $C^1$ perturbations.
\end{proof}

Let
\[
  \theta=\frac{\xi}{|\xi|},
  \qquad
  \Pi_\theta v=v-(v\cdot\theta)\theta
\]
be the orthogonal projection onto $\theta^\perp$.

\begin{proposition}\label{prop:conductivity-first-invariant}
Let
\[
  b(\omega,x)=\nabla\log a(\omega,x),
  \qquad
  \overline b(x)=\frac{\nabla m(x)}{m(x)},
\]
and define the positive semidefinite matrix
\begin{equation}\label{eq:conductivity-covariance-matrix}
  C_a(x)=\E\left[a(\omega,x)
  (b(\omega,x)-\overline b(x))\otimes
  (b(\omega,x)-\overline b(x))\right].
\end{equation}
The order $-4$ symbol of $\overline R-R_{\gamma_h}$ determines, for every $\theta\in S^{n-1}$,
\begin{equation}\label{eq:conductivity-invariant}
  \mathcal V_\gamma(x,\theta)
  =\tr C_a(x)-\theta^TC_a(x)\theta.
\end{equation}
Equivalently,
\[
  \mathcal V_\gamma(x,\theta)
  =\E\left[\frac{|\Pi_\theta\nabla a(\cdot,x)|^2}{a(\cdot,x)}\right]
  -\frac{|\Pi_\theta\nabla m(x)|^2}{m(x)}.
\]
More precisely,
\[
  \mathcal V_\gamma(x,\theta)
  =- |\xi|^4\sigma_{-4}(\overline R-R_{\gamma_h})(x,\xi),
  \qquad \theta=\xi/|\xi|.
\]
If $n\ge2$, knowing $\mathcal V_\gamma(x,\theta)$ for all $\theta\in S^{n-1}$ determines the whole matrix $C_a(x)$.
\end{proposition}

\begin{proof}
For the localized interior symbol we use the Kohn--Nirenberg convention
$D_x=(1/i)\partial_x$ and write the parametrix symbol as
$r\sim r_{-2}+r_{-3}+r_{-4}+\cdots$.  For a deterministic
$a=\gamma^{-1}$ the symbol of $-\nabla\cdot\gamma\nabla$ is
\[
  p=p_2+p_1,
  \qquad
  p_2=a^{-1}|\xi|^2,
  \qquad
  p_1=ia^{-2}\nabla a\cdot\xi .
\]
The equation $p\# r=1$ gives
\[
  \sum_\alpha\frac1{\alpha!}\partial_\xi^\alpha p\,D_x^\alpha r=1.
\]
At the first two orders one obtains
\[
  r_{-2}=a|\xi|^{-2},
  \qquad
  r_{-3}=i\nabla a\cdot\xi\,|\xi|^{-4}.
\]
At the next order the terms of total degree $-2$ in $p\# r$ give
\begin{equation}\label{eq:conductivity-rminus4-recursion}
\begin{split}
0={}&p_2r_{-4}+p_1r_{-3}
 +\sum_j(\partial_{\xi_j}p_2)D_{x_j}r_{-3}
 +\sum_j(\partial_{\xi_j}p_1)D_{x_j}r_{-2}  \\
&+\frac12\sum_{i,j}(\partial_{\xi_i}\partial_{\xi_j}p_2)
  D_{x_i}D_{x_j}r_{-2}.
\end{split}
\end{equation}
Substituting $r_{-2}$ and $r_{-3}$ into \eqref{eq:conductivity-rminus4-recursion} yields
\[
\begin{split}
 p_1r_{-3}&=-\frac{(\nabla a\cdot\xi)^2}{a^2|\xi|^4},\qquad
 \sum_j(\partial_{\xi_j}p_2)D_{x_j}r_{-3}
 =\frac{2\nabla^2a(\xi,\xi)}{a|\xi|^4},\\
 \sum_j(\partial_{\xi_j}p_1)D_{x_j}r_{-2}
 &=\frac{|\nabla a|^2}{a^2|\xi|^2},\qquad
 \frac12\sum_{i,j}(\partial_{\xi_i}\partial_{\xi_j}p_2)
  D_{x_i}D_{x_j}r_{-2}
 =-\frac{\Delta a}{a|\xi|^2}.
\end{split}
\]
Since $p_2^{-1}=a|\xi|^{-2}$, this gives
\begin{align*}
 \sigma_{-4}(R_\gamma)(x,\xi)=r_{-4}
 &=\frac{\Delta a}{|\xi|^4}
 -\frac{2\nabla^2a(\xi,\xi)}{|\xi|^6}
 +\frac{(\nabla a\cdot\xi)^2}{a|\xi|^6}
 -\frac{|\nabla a|^2}{a|\xi|^4}.
\end{align*}
The first two terms in $\sigma_{-4}(R_\gamma)$ are linear in $a$ and cancel when the averaged symbol is compared with the symbol obtained by replacing $a$ with $m=\E a$.  The remaining terms give
\[
\begin{split}
- |\xi|^4\sigma_{-4}(\overline R-R_{\gamma_h})(x,\xi)
&=\E\left[\frac{|\nabla a|^2}{a}
     -\frac{(\nabla a\cdot\theta)^2}{a}\right] \\
&\quad -\left(\frac{|\nabla m|^2}{m}
     -\frac{(\nabla m\cdot\theta)^2}{m}\right),
\end{split}
\]
which is the displayed expression involving $\Pi_\theta$.  Since $\nabla a=ab$ and $\nabla m=m\overline b$, this expression is also
\[
  \tr C_a(x)-\theta^TC_a(x)\theta.
\]
The matrix $C_a(x)$ is positive semidefinite by definition.  Finally, if $n\ge2$, then averaging \eqref{eq:conductivity-invariant} over $S^{n-1}$ determines $\tr C_a(x)$, since
\[
  \int_{S^{n-1}}\theta^TC_a(x)\theta\,d\theta
  =\frac1n\tr C_a(x)
\]
for normalized surface measure.  Hence $\theta^TC_a(x)\theta=\tr C_a(x)-\mathcal V_\gamma(x,\theta)$ is known for all $\theta$, and this determines $C_a(x)$.
\end{proof}

\begin{corollary}\label{cor:conductivity-Green-rigidity}
Assume $n\ge2$ and $\OmegaD$ is connected.  Let $\gamma,\gamma_1,\gamma_2\in\mathcal G_+^\infty(\OmegaD)$ and let $0<\tau<1$.  If
\[
  R_\gamma=\tau R_{\gamma_1}+(1-\tau)R_{\gamma_2},
\]
then there is a constant $C>0$ such that
\[
  \gamma_2=C\gamma_1
\]
and
\[
  \gamma^{-1}=\tau\gamma_1^{-1}+(1-\tau)\gamma_2^{-1}.
\]
\end{corollary}

\begin{proof}
Let $a_j=\gamma_j^{-1}$.  Comparing principal symbols gives
\[
  \gamma^{-1}=\tau a_1+(1-\tau)a_2=:m.
\]
Hence the left side has the same first two symbol terms as $R_{m^{-1}}$.  Comparing the order $-4$ symbols and using Proposition~\ref{prop:conductivity-first-invariant} for the two-point random variable taking the values $a_1$ and $a_2$ gives $\mathcal V_\gamma(x,\theta)=0$ for all $\theta$.  In this two-point case
\[
  C_a(x)=\frac{\tau(1-\tau)a_1(x)a_2(x)}{m(x)}
  \nabla\log\frac{a_1}{a_2}(x)
  \otimes
  \nabla\log\frac{a_1}{a_2}(x).
\]
Since the coefficient in front is positive, $\mathcal V_\gamma(x,\theta)=0$ for all $\theta$ implies
\[
  \Pi_\theta\nabla\log(a_1/a_2)=0
\]
for every $\theta\in S^{n-1}$.  Since $n\ge2$, this forces $\nabla\log(a_1/a_2)=0$.  Connectedness gives $a_1/a_2$ constant, equivalently $\gamma_2=C\gamma_1$.
\end{proof}

\begin{remark}
The exact analogue of the moment recovery theorem for $G_q$ is false for $R_\gamma$.  If $\gamma(\omega,x)=\alpha(\omega)\gamma_0(x)$, then
\[
  R_{\gamma(\omega)}=\alpha(\omega)^{-1}R_{\gamma_0},
  \qquad
  \E R_\gamma=\E[\alpha^{-1}]R_{\gamma_0}
  =R_{\gamma_0/\E[\alpha^{-1}]}.
\]
Thus random scalar multiples are indistinguishable from a deterministic harmonic-mean conductivity by the averaged Green's operator.
\end{remark}

\section*{Acknowledgements}
The author was supported by the National Science and Technology Council (NSTC) grant number 113-2115-M-A49-018-MY3.

\bibliographystyle{plain}
\bibliography{stochastic_inverse_coefficients_namechecked_v2}

\end{document}